\documentclass[12pt]{amsart}

\usepackage{amsmath}

\usepackage{amssymb}

\usepackage{mathtools}

\usepackage{array}

\usepackage{a4wide}
  
\usepackage{stmaryrd}

\usepackage{tikz}

\usepackage{tikz-cd}
\usetikzlibrary{cd}

\usepackage{bm}

\usepackage{mathrsfs}

\usepackage{amsthm}

\usepackage{enumerate}

\usepackage{fancyhdr}

\usepackage{hyperref}

\usepackage[inline]{enumitem}

\usepackage{marginnote}

\theoremstyle{plain}

\newtheorem{thm}{Theorem}[section]

\newtheorem{lem}[thm]{Lemma}

\newtheorem{cor}[thm]{Corollary}

\newtheorem{prop}[thm]{Proposition}

\newtheorem{deflem}[thm]{Definition/Lemma}

\theoremstyle{definition}

\newtheorem{rem2}[thm]{Remark}

\newtheorem*{note}{Note}

\newtheorem{defn}[thm]{Definition}

\newtheorem{eg}[thm]{Example}

\newcommand{\bC}{\mathbb{C}}

\newcommand{\bQ}{\mathbb{Q}}

\newcommand{\bZ}{\mathbb{Z}}

\newcommand{\fx}{\mathfrak{x}}

\newcommand{\scD}{\mathscr{D}}

\newcommand{\scI}{\mathscr{I}}

\newcommand{\scM}{\mathscr{M}}

\newcommand{\scO}{\mathscr{O}}

\begin{document}

\title{Hodge filtrations for twisted parametrically prime divisors}

\author{Henry Dakin}

\date{}

\maketitle

\begin{abstract}
We extend previous work to the case of $\bQ$-divisors. Namely, for certain parametrically prime holomorphic functions $f$ and $\alpha \geq 0$, we obtain an explicit expression for the Hodge filtration on $\scM(f^{-\alpha}):=\scO_X(*f)f^{-\alpha}$.
\end{abstract}

\tableofcontents

\section{Introduction}\label{intro}

This article extends the main results of \cite{BD24}. We generalise the main results of this paper to the case of $\bQ$-divisors. All of the proofs are completely analogous to those in \cite{BD24}.

Namely, we are interested in calculating the canonical Hodge filtration $F_{\bullet}^H$ associated to the left $\scD_X$-module $\scM(f^{-\alpha}):=\scO_X(*f)f^{-\alpha}$, where $X$ is a complex manifold, $\alpha\in\bQ_{\geq 0}$ and $f\in\scO_X$ is a non-invertible non-zero reduced global section. This Hodge filtration is the one induced by the natural complex mixed Hodge module structure overlying $\scM(f^{-\alpha})$. In this article we assume further that $f$ is parametrically prime (see \cite{BD24}, Definition 4.14) and Euler homogeneous (\cite{BD24}, Definition 4.15), and that the set of roots $\rho_{f,\hspace{0.7pt}\fx}$ of the local Bernstein-Sato polynomial $b_{f,\hspace{0.7pt}\fx}(s)$ of $f$ at $\fx$ satisfies $\rho_{f,\hspace{0.7pt}\fx}\subseteq(-2-\alpha,-\alpha)$, for any $\fx\in X$. We then generalise the results of \cite{BD24}, which covers the case $\alpha =0$. The main theorem is the following.

\begin{thm}

Let $X$ be a complex manifold, $\fx\in X$ and $\alpha \in \bQ_{\geq 0}$. Let $f\in\scO_X$ be reduced, non-invertible, non-zero. Assume that $f$ is of Euler homogeneous and parametrically prime at $\fx$ and that $\rho_{f,\hspace{0.7pt}\fx}\subseteq (-2-\alpha,-\alpha)$. Then 
\[F_k^H\scM(f^{-\alpha})_{\fx}=\phi_{-\alpha}(\Gamma_{f,-\alpha,\hspace{0.7pt}\fx}\cap F_k^{\sharp}\scD_{X,\hspace{0.7pt}\fx}[s])\cdot f_{\fx}^{-1-\alpha},\]
where
\[\Gamma_{f,-\alpha,\hspace{0.7pt}\fx}:=\scD_{X,\hspace{0.7pt}\fx}[s]f_{\fx} +\scD_{X,\hspace{0.7pt}\fx}[s]\beta_{f,-\alpha,\hspace{0.7pt}\fx}(-s)+ \text{\emph{ann}}_{\scD_{X,\hspace{0.7pt}\fx}[s]}f_{\fx}^{s-1}\subseteq \scD_{X,\hspace{0.7pt}\fx}[s]\]
and
\[\phi_{-\alpha}:\scD_{X,\hspace{0.7pt}\fx}[s]\to \scD_{X,\hspace{0.7pt}\fx}\, ; \, P(s) \mapsto P(-\alpha)\]
and 
\[\beta_{f,-\alpha,\hspace{0.7pt}\fx}(s):=\prod_{\lambda \in \rho_{f,\hspace{0.7pt}\fx}\cap (-\alpha-1,-\alpha)}(s+\lambda+1)^{l_{\lambda}},\]
$l_{\lambda}$ being the multiplicity of $\lambda$ as a root of $b_{f,\hspace{0.7pt}\fx}(s)$.

\label{thmmain}

\end{thm}

\begin{note}

Thus giving us expressions for the Hodge ideals associated to the $\bQ$-divisor $D=\alpha H$, $H=\text{div}(f)$ (for $\alpha\neq 0$).

\end{note}

The statement for $k=0$ actually holds under much weaker assumptions.

\begin{thm}

Let $X$ be a complex manifold, $\fx\in X$ and $\alpha \in \bQ_{\geq 0}$. Let $f\in\scO_X$ be reduced, non-invertible, non-zero. Assume that $\rho_{f,\hspace{0.7pt}\fx}\subseteq (-2-\alpha,-\alpha)$. Then
\[F_0^H\scM(f^{-\alpha})_{\fx}=(\Gamma_{f,-\alpha,\hspace{0.7pt}\fx}\cap \scO_{X,\hspace{0.7pt}\fx})\cdot f_{\fx}^{-1-\alpha}.\]
In particular the multiplier ideal sheaf $\scI(f,\alpha-\epsilon)=I_0(\alpha H)$, $0<\epsilon<<1$, is given locally at $\fx$ by $\Gamma_{f,-\alpha,\hspace{0.7pt}\fx}\cap\scO_{X,\hspace{0.7pt}\fx}$.

\label{thmmain2}
    
\end{thm}

\noindent We use a completely analogous proof to that used in \cite{BD24}. Namely, we have an expression for certain $V$-filtration steps on the direct image along the graph embedding $i_{f,+}\scM(f^{-\alpha})_{(\fx,0)}$ under our assumption on $\rho_{f,\hspace{0.7pt}\fx}$, and then under our further assumptions of parametrically prime and Euler homogeneous, we obtain some compatibility between the order filtration and induced $V$-filtration on $\scM(f^{-\alpha})_{\fx}$, which we use to understand how the $t$-order filtration on $i_{f,+}\scM(f^{-\alpha})_{(\fx,0)}$ acts when restricted to the $V$-filtration.

The structure of the paper is as follows. In Section \ref{sectiondefn} we define the main objects and properties used in the rest of the paper. Section \ref{sectionMfalpha} is then concerned with some basic properties of the modules $\scM(f^{-\alpha})$ and the filtrations on these modules. Section \ref{sectionmainthm} deals with the proof of the main theorem.

\section{Definitions} \label{sectiondefn}

Let $X$ be a complex manifold of dimension $n$, $f \in \scO_X$ such that the zero locus $\text{Var}(f)\subset X$ of $f$ defines a hypersurface of $X$ and such that $f$ is reduced. Let $\fx\in X$. Write $f_{\fx}\in\scO_{X,\hspace{0.7pt}\fx}$ for the germ of $f$ at $\fx$.

\begin{defn} \label{def-parametricallyPrime}
    Consider $\text{ann}_{\scD_{X,\hspace{0.7pt}\fx}[s]} f_{\fx}^{s-1}$,  the $\scD_{X,\hspace{0.7pt}\fx}[s]$-annihilator of $f_{\fx}^{s-1}$. We write $F_{\bullet}^{\sharp}$ for the total order filtration on $\scD_{X,\hspace{0.7pt}\fx}[s]$, defined by
    \[F_k^{\sharp}\scD_{X,\hspace{0.7pt}\fx}[s] := \sum_{i\geq 0}F_{k-i}\scD_{X,\hspace{0.7pt}\fx}s^i.\]
    We say that $f$ is \emph{parametrically prime at $\fx \in X$} when $\text{gr}^{F^{\sharp}} (\text{ann}_{\scD_{X,\hspace{0.7pt}\fx}[s]} f_{\fx}^{s-1})$ is a prime $(\text{gr}^{F^{\sharp}} \scD_{X,\hspace{0.7pt}\fx}[s]$)-ideal. 
\end{defn}

\begin{defn} \label{def-EulerHom} \enspace 
    An \emph{Euler vector field} $E$ of $f$ at $\fx$ is a derivation $E \in \text{Der}_{\bC}(\scO_{X,\hspace{1pt}\fx})$ such that $E \cdot f_{\fx} = f_{\fx}$. We say that $f$ is \emph{Euler homogeneous} at $\fx$ if it admits an Euler vector field at $\fx$.
    We say that $f$ is \emph{Euler homogeneous} if it is Euler homogeneous at all $\fx \in \text{Var}(f)$.
\end{defn}

There are many classes of examples satisfying these properties:

\begin{eg}

\begin{enumerate}[label=\roman*)]

\item If $f$ is of linear Jacobian type at $\fx$ (see \cite{BD24}, Definition 5.3), then it is parametrically prime and Euler homogeneous at $\fx$ (by \cite{BD24}, Propositions 5.4 and 5.9, and Proposition 4.17).

\item If $f$ is tame, strongly Euler homogeneous and Saito holonomic at $\fx$, then it is of linear Jacobian type (see \cite{Wal17}).

\item Strongly Koszul free divisors satisfy the conditions in ii), see \cite{Nar15}.

\item Any positively weighted homogeneous locally everywhere divisor is Saito holonomic. Thus for instance all tame hyperplane arrangements satisfy the conditions in ii). And if $n=3$, any positively weighted homogeneous locally everywhere divisor satisfies the conditions in ii).
    
\end{enumerate}
    
\end{eg}

\begin{defn}

We write $b_{f,\hspace{0.7pt}\fx}(s)\in\bC[s]$ for the \emph{local Bernstein-Sato polynomial of $f$ at $\fx$}, defined to be the monic polynomial of minimal degree satisfying
\[b_{f,\hspace{0.7pt}\fx}(s)\cdot f_{\fx}^s \in \scD_{X,\hspace{0.7pt}\fx}[s]\cdot f_{\fx}^{s+1}.\]
We write $\rho_{f,\hspace{0.7pt}\fx}$ for the set of roots of $b_{f,\hspace{0.7pt}\fx}(s)$. Then $-1 \in \rho_{f,\hspace{0.7pt}\fx}$ and $\rho_{f,\hspace{0.7pt}\fx}\subseteq \bQ\cap(-n,0)$.
    
\end{defn}

\section{$\scM(f^{-\alpha})$ and the $V$-filtration} \label{sectionMfalpha}

We investigate how the module $\scM(f^{-\alpha})$ is related to the (loca) Bernstein-Sato polynomial of $f$. We then obtain a local expression for the Kashiwara-Malgrange $V$-filtration on $i_{f,+}\scM(f^{-\alpha})$ along $\{t=0\}$ under the assumption $\rho_{f,\hspace{0.7pt}\fx}\subseteq(-2-\alpha,-\alpha)$. Finally, we give a formula for the Hodge filtration on $\scM(f^{-\alpha})$ in terms of this Kashiwara-Malgrange $V$-filtration. 

We write $i_f:X\to X\times\bC_t\, ; \, \fx \to (\fx,f(\fx))$ for the graph embedding of the function $f$.

\begin{prop}

The morphism 
\[\Phi:i_{f,+}\scM(f^{-\alpha})\to i_{f,+}\scO_X(*f)\, ; \, \sum_{i=0}^kh_if^{-\alpha}\partial_t^i \to \sum_{i=0}^kg_i\partial_t^i,\]
where $g_i$ are determined uniquely through the equality
\[\sum_{i=0}^kh_if^{-i}Q_i(-s-\alpha) = \sum_{i=0}^kg_if^{-i}Q_i(-s) \,\text{ in }\, \scO_X(*f)[s],\]
and where $Q_i(x):=\prod_{j=0}^{i-1}(x+j),$ is an isomorphism of $\scD_X\langle t,t^{-1},s\rangle$-modules, when $s$ acts as $s+\alpha$ on the right hand side of this map. Moreover, 
\[\Phi(V^{\gamma}i_{f,+}\scM(f^{-\alpha})) = V^{\gamma+\alpha}i_{f,+}\scO_X(*f).\]

\label{propiso}
    
\end{prop}

\begin{proof}

See \cite{MP20}, Proposition 2.6.
    
\end{proof}

\begin{lem}

If $f$ is Euler homogeneous at $\fx$ and $\rho_{f,\hspace{0.7pt}\fx}\,\cap\,\left( \bZ_{<-1}-\alpha\right) = \emptyset$, then \[\text{\emph{ann}}_{\scD_{X,\hspace{0.7pt}\fx}}f_{\fx}^{-1-\alpha}=\text{\emph{ann}}_{\scD_{X,\hspace{0.7pt}\fx}}f_{\fx}^s+\scD_{X,\hspace{0.7pt}\fx}(E+1+\alpha),\] 
where $E$ is an Euler vector field for $f$ at $\fx$. Moreover,
\[\text{\emph{ann}}_{\scD_{X\times\bC_t,(\fx,0)}}f_{\fx}^{-1-\alpha}=\scD_{X\times\bC_t,(\fx,0)}\text{\emph{ann}}_{\scD_{X,\hspace{0.7pt}\fx}}f_{\fx}^s+\scD_{X\times\bC_t,(\fx,0)}(E+\partial_tt+1+\alpha)+\scD_{X\times\bC_t,(\fx,0)}(t-f_{\fx}),\]
when $f_{\fx}^{-1-\alpha}$ is viewed as an element of $i_{f,+}\scM(f^{-\alpha})_{(\fx,0)} \simeq \sum_{i\geq 0}\scM(f^{-\alpha})_{\fx}\partial_t^i$.

\label{lemann}
    
\end{lem}

\begin{proof}

Firstly we show that 
\[\text{ann}_{\scD_{X,\hspace{0.7pt}\fx}[s]}f_{\fx}^s=\text{ann}_{\scD_{X,\hspace{0.7pt}\fx}}f_{\fx}^s+\scD_{X,\hspace{0.7pt}\fx}[s](E-s).\] 
$\supseteq$ is clear. If $P\in \text{ann}_{\scD_{X,\hspace{0.7pt}\fx}[s]}f_{\fx}^s$, then we may write 
\[P=\sum_{i\geq 0}P_is^i, \,\,\, P_i\in\scD_{X,\hspace{0.7pt}\fx}.\]
For $i\geq 1$, $E^i-s^i \in \scD_{X,\hspace{0.7pt}\fx}[s](E-s)$, so
\[P-\sum_{i\geq 0}P_iE^i \in \scD_{X,\hspace{0.7pt}\fx}[s](E-s).\]
Moreover, $\sum_{i\geq 0}P_is^i\in \text{ann}_{\scD_{X,\hspace{0.7pt}\fx}[s]}f_{\fx}^s$ implies that $\sum_{i\geq 0}P_iE^i\in \text{ann}_{\scD_{X,\hspace{0.7pt}\fx}}f_{\fx}^s$, so 
\[P\in \text{ann}_{\scD_{X,\hspace{0.7pt}\fx}}f_{\fx}^s+\scD_{X,\hspace{0.7pt}\fx}[s](E-s)\]
as required.

Now, Proposition 6.2 of \cite{Kash76} says that under our assumptions
\[\text{ann}_{\scD_{X,\hspace{0.7pt}\fx}}f_{\fx}^{-1-\alpha} = (\text{ann}_{\scD_{X,\hspace{0.7pt}\fx}[s]}f_{\fx}^s+\scD_{X,\hspace{0.7pt}\fx}[s](s+\alpha+1))\cap\scD_{X,\hspace{0.7pt}\fx},\]
from which the first statement in the lemma now follows.

Again, it is easy to check the reverse inclusion for the second statement. Let $P\in \text{ann}_{\scD_{X\times\bC_t,(\fx,0)}}f_{\fx}^{-1-\alpha}$. Without loss of generality we may write $P=\sum_iP_i\partial_t^{\,i}$, $P_i\in\scD_{X,\hspace{0.7pt}\fx}$. Write $k$ for the largest $i$ for which $P_i\neq 0$. We prove the required inclusion via induction on $k$. 

If $k=0$, then $P\in\scD_{X,\hspace{0.7pt}\fx}$. Considering the coefficient of $\partial_t^0$ in $P\cdot f_{\fx}^{-1-\alpha}$, we see that $P\cdot f_{\fx}^{-1-\alpha}=0$ in $\scM(f^{-\alpha})_{\fx}$ as well, implying by the above proof that 
\[P\in \text{ann}_{\scD_{X,\hspace{0.7pt}\fx}}f_{\fx}^s+\scD_{X,\hspace{0.7pt}\fx}(E+\alpha+1).\]
Write $P=Q+R(E+1+\alpha)$, $Q\in\text{ann}_{\scD_{X,\hspace{0.7pt}\fx}}f_{\fx}^s, R\in\scD_{X,\hspace{0.7pt}\fx}$. Then, since $E+\partial_tt+1+\alpha\in\text{ann}_{\scD_{X\times\bC_t,(\fx,0)}}f_{\fx}^{-1-\alpha}$, we have that $R\partial_tt\in\text{ann}_{\scD_{X\times\bC_t,(\fx,0)}}f_{\fx}^{-1-\alpha}$, implying in turn that $R\in\text{ann}_{\scD_{X\times\bC_t,(\fx,0)}}f_{\fx}^{-1-\alpha}$. Thus we may repeatedly iterate the above procedure to see that 
\[P\in\text{ann}_{\scD_{X,\hspace{0.7pt}\fx}}f_{\fx}^s+\scD_{X,\hspace{0.7pt}\fx}(E+\alpha+1)^k \,\,\,\text{ for all }\,\, k \geq 1.\]
In particular,
\[P\cdot f_{\fx}^s \in \scD_{X,\hspace{0.7pt}\fx}\cdot (s+\alpha+1)^kf_{\fx}^s \,\,\,\text{ for all }\,\,k\geq 1.\]
But this is only possible if $P\cdot f_{\fx}^s=0$, so in fact
\[P\in \text{ann}_{\scD_{X,\hspace{0.7pt}\fx}}f_{\fx}^s.\]

If instead $k>0$, then identically to above we have that
\[P_0\in \text{ann}_{\scD_{X,\hspace{0.7pt}\fx}}f_{\fx}^s+\scD_{X,\hspace{0.7pt}\fx}(E+\alpha+1),\]
so, if we write $P_0=Q+R(E+1+\alpha)$ with $Q\in\text{ann}_{\scD_{X,\hspace{0.7pt}\fx}}f_{\fx}^s$ and $R\in\scD_{X,\hspace{0.7pt}\fx}$, then
\[P-P_0-R\partial_tt\in\text{ann}_{\scD_{X\times\bC_t,(\fx,0)}}f_{\fx}^{-1-\alpha}.\]
But this then implies that 
\[\sum_{i\geq 0}P_{i+1}\partial_t^i-Rt\in\text{ann}_{\scD_{X\times\bC_t,(\fx,0)}}f_{\fx}^{-1-\alpha}.\]
This operator has order $k-1$, so we may apply our inductive hypothesis to see that 
\[\sum_{i\geq 0}P_{i+1}\partial_t^i-Rt\in\scD_{X\times\bC_t,(\fx,0)}\text{ann}_{\scD_{X,\hspace{0.7pt}\fx}}f_{\fx}^s+\scD_{X\times\bC_t,(\fx,0)}(E+\partial_tt+1+\alpha)+\scD_{X\times\bC_t,(\fx,0)}(t-f_{\fx}),\]
which of course then implies the same for $P-P_0-R\partial_tt$. Since
\[P=(P-P_0-R\partial_tt)+(Q+R(E+\partial_tt+1+\alpha)),\]
we conclude as required that 
\[P\in\scD_{X\times\bC_t,(\fx,0)}\text{ann}_{\scD_{X,\hspace{0.7pt}\fx}}f_{\fx}^s+\scD_{X\times\bC_t,(\fx,0)}(E+\partial_tt+1+\alpha)+\scD_{X\times\bC_t,(\fx,0)}(t-f_{\fx}),\]

\end{proof}

\begin{lem}

$k \in \bZ$. $\scM(f^{-\alpha})_{\fx}$ is generated by $f_{\fx}^{-\alpha-k}$ if and only if $\rho_{f,\hspace{0.7pt}\fx}\cap\left( \bZ_{<-k}-\alpha\right) = \emptyset$.

\label{lemgen}
    
\end{lem}

\begin{proof}

\underline{$\Leftarrow$} It suffices to prove that $f_{\fx}^{-\alpha-l}\in\scD_{X,\hspace{0.7pt}\fx}\cdot f_{\fx}^{-\alpha-k}$ for all $l \geq k$. For $l=k$ this is trivial, so assume $l>k$. We prove the result inductively. 

Choose $P(s)\in\scD_{X,\hspace{0.7pt}\fx}[s]$ such that
\[P(s)f_{\fx}^{s+1}=b_{f,\hspace{0.7pt}\fx}(s)f_{\fx}^s.\]
Then $b_{f,\hspace{0.7pt}\fx}(-\alpha-l)\neq 0$ by hypothesis, so 
\[f_{\fx}^{-\alpha-l}=b_{f,\hspace{0.7pt}\fx}(-\alpha-l)^{-1}P(-\alpha-l)f_{\fx}^{-\alpha-(l-1)} \in \scD_{X,\hspace{0.7pt}\fx}\cdot f_{\fx}^{-\alpha-k}\]
by the inductive hypothesis, showing the result.

\underline{$\Rightarrow$} Assume for contradiction that $b_{f,\hspace{0.7pt}\fx}(-\alpha-k)=0$ but also that $\rho_{f,\hspace{0.7pt}\fx}\cap \left(\bZ_{< -k}-\alpha\right) = \emptyset$ and that $\scM(f^{-\alpha})_{\fx}$ is generated by $f_{\fx}^{-\alpha-k+1}$. Then, by the above proof, $\scD_{X,\hspace{0.7pt}\fx}\cdot f_{\fx}^{-\alpha-k}=\scD_{X,\hspace{0.7pt}\fx}\cdot f_{\fx}^{-\alpha-k+1}$. This implies that 
\[\scD_{X,\hspace{0.7pt}\fx}=\scD_{X,\hspace{0.7pt}\fx}f_{\fx}+\text{ann}_{\scD_{X,\hspace{0.7pt}\fx}}f_{\fx}^{-\alpha-k},\]
implying, by using Proposition 6.2 of \cite{Kash76}, that 
\[\scD_{X,\hspace{0.7pt}\fx}[s]=\scD_{X,\hspace{0.7pt}\fx}[s]f_{\fx}+\text{ann}_{\scD_{X,\hspace{0.7pt}\fx}[s]}f_{\fx}^s + \scD_{X,\hspace{0.7pt}\fx}[s](s+\alpha+k).\]
Therefore
\[\frac{b_{f,\hspace{0.7pt}\fx}(s)}{s+\alpha+k}\in \scD_{X,\hspace{0.7pt}\fx}[s]f_{\fx}+\text{ann}_{\scD_{X,\hspace{0.7pt}\fx}[s]}f_{\fx}^s + \scD_{X,\hspace{0.7pt}\fx}[s]b_{f,\hspace{0.7pt}\fx}(s),\]
implying that 
\[\frac{b_{f,\hspace{0.7pt}\fx}(s)}{s+\alpha+k}f_{\fx}^s \in \scD_{X,\hspace{0.7pt}\fx}[s]f_{\fx}^{s+1},\]
contradicting the minimality of the Bernstein-Sato polynomial $b_{f,\hspace{0.7pt}\fx}(s)$.
    
\end{proof}

\begin{rem2}

In the situation of the above lemma, it is then easy to show that $i_{f,+}\scM(f^{-\alpha})_{(\fx,0)}$ is also generated by $f_{\fx}^{-\alpha-k}$.
    
\end{rem2}

\begin{defn}

Assume that $\rho_{f,\hspace{0.7pt}\fx}\cap\left(\bZ_{<-1}-\alpha\right)=\emptyset$. Then the \emph{induced $V$-filtration} on $i_{f,+}\scM(f^{-\alpha})_{(\fx,0)}$ is defined to be
\[V^k_{\text{ind}}i_{f,+}\scM(f^{-\alpha})_{(\fx,0)}:=V^k\scD_{X\times\bC, \,(\fx,0)}\cdot f_{\fx}^{-1-\alpha}.\]
The \emph{order filtration} on $\scM(f^{-\alpha})_{\fx}$ is defined to be
\[F_k^{\text{ord}}\scM(f^{-\alpha})_{\fx}:=F_k\scD_{X, \,\hspace{0.7pt}\fx}\cdot f_{\fx}^{-1-\alpha}.\]
The \emph{order filtration} on $i_{f,+}\scM(f^{-\alpha})_{(\fx,0)}$ is defined to be
\[F_k^{\text{ord}}i_{f,+}\scM(f^{-\alpha})_{(\fx,0)}:=F_k\scD_{X\times\bC, \,(\fx,0)}\cdot f_{\fx}^{-1-\alpha}.\]
    
\end{defn}

\begin{lem}

The Bernstein-Sato polynomial associated to $V^{\bullet}_{\text{\emph{ind}}}i_{f,+}\scM(f^{-\alpha})_{(\fx,0)}$ is $b_{f,\hspace{0.7pt}\fx}(-s-1-\alpha)$.

\label{lemBfcn}
    
\end{lem}

\begin{proof}

It is clear that $V_{\text{ind}}^{\bullet}i_{f,+}\scM(f^{-\alpha})_{(\fx,0)}$ is a good $V$-filtration. Therefore it has some non-zero B-function, monic of minimal degree (the \emph{Bernstein-Sato polynomial} associated to the good $V$-filtration $V_{\text{ind}}^{\bullet}i_{f,+}\scM(f^{-\alpha})_{(\fx,0)}$), which we will denote $b_{V_{\text{ind}}^{\bullet}, \,\hspace{0.7pt}\fx}(s)$.

Firstly, consider the map
\[\varphi_1:i_{f,+}\scM(f^{-\alpha}) \to \scM(f^{-\alpha})[s]f^s \, ; \, uf^{-\alpha}\partial_t^i \mapsto uf^{-i-\alpha}Q_i(-s)f^s.\]
This is an isomorphism of $\scD_X\langle t,t^{-1},s\rangle$-modules, where $t$ acts as $s\mapsto s+1$ on the right hand side (see Proposition 2.5 of \cite{MP20}). Due to this linearity, we have for $k\geq 0$ that
\[\varphi_1\left( V^k_{\text{ind}}i_{f,+}\scM(f^{-\alpha})_{(\fx,0)}\right) = \scD_{X,\hspace{0.7pt}\fx}[s]\cdot f_{\fx}^{s+k-1-\alpha},\]
implying that
\[\text{gr}_{V_{\text{ind}}}^ki_{f,+}\scM(f^{-\alpha})_{(\fx,0)} \simeq \frac{\scD_{X,\hspace{0.7pt}\fx}[s]\cdot f_{\fx}^{s+k-1-\alpha}}{\scD_{X,\hspace{0.7pt}\fx}[s]\cdot f_{\fx}^{s+k+1-1-\alpha}}\]
for $k \geq 0$. By definition, the $\bC[s]$-annihilator of this module is generated by $b_{f,\hspace{0.7pt}\fx}(s+k-1-\alpha)$. Remembering that $s$ is acting as $-\partial_tt$ on the left hand side of the above isomorphism, we conclude that
\[b_{f,\hspace{0.7pt}\fx}(-s-1-\alpha)\, \big| \, b_{V_{\text{ind}}^{\bullet},\hspace{0.7pt}\fx}(s),\]
since $b_{V_{\text{ind}},\hspace{0.7pt}\fx}(\partial_tt-k)$ annihilates $\text{gr}_{V_{\text{ind}}}^ki_{f,+}\scM(f^{-\alpha})_{(\fx,0)}$ by definition. 

So it suffices now to prove that $b_{f,\hspace{0.7pt}\fx}(-\partial_tt+k-1-\alpha)$ also annihilates $\text{gr}_{V_{\text{ind}}}^ki_{f,+}\scM(f^{-\alpha})_{(\fx,0)}$ when $k < 0$, as this will show conversely that 
\[b_{V_{\text{ind},\hspace{0.7pt}\fx}^{\bullet}}(s)\, \big| \, b_{f,\hspace{0.7pt}\fx}(-s-1-\alpha).\]

Now by the definition of the induced $V$-filtration,
\[\varphi_1\left( V^k_{\text{ind}}i_{f,+}\scM(f^{-\alpha})_{(\fx,0)}\right) = \varphi_1\left( V^{k+1}_{\text{ind}}i_{f,+}\scM(f^{-\alpha})_{(\fx,0)}\right)+\scD_{X,\hspace{0.7pt}\fx}[s]s^{-k}f_{\fx}^{s+k-1-\alpha}\]
whenever $k \leq -1$. Thus
\begin{align*}
\varphi_1(b_{f,\hspace{0.7pt}\fx}(-\partial_tt+k-1-\alpha)\cdot V^k_{\text{ind}}&i_{f,+}\scM(f^{-\alpha})_{(\fx,0)}) \\&= \varphi_1\left(b_{f,\hspace{0.7pt}\fx}(-\partial_tt+k-1-\alpha)\cdot V^{k+1}_{\text{ind}}i_{f,+}\scM(f^{-\alpha})_{(\fx,0)}\right)  \\& \;\;\;\;+s^{-k+1}b_{f,\hspace{0.7pt}\fx}(s+k-1-\alpha)\cdot \scD_{X,\hspace{0.7pt}\fx}[s]\cdot f_{\fx}^{s+k-1-\alpha}\\
& \subseteq \varphi_1\left(V^{k+1}_{\text{ind}}i_{f,+}\scM(f^{-\alpha})_{(\fx,0)}\right) + s^{-k+1} \scD_{X,\hspace{0.7pt}\fx}[s]\cdot f_{\fx}^{s+k-\alpha}\\
& \subseteq \varphi_1\left(V^{k+1}_{\text{ind}}i_{f,+}\scM(f^{-\alpha})_{(\fx,0)}\right),
\end{align*}
which of course in turn implies that 
\[b_{f,\hspace{0.7pt}\fx}(-\partial_tt+k-1-\alpha)\cdot \text{gr}_{V_{\text{ind}}}^ki_{f,+}\scM(f^{-\alpha})_{(\fx,0)}=0,\]
as required.
    
\end{proof}

\begin{prop}

If $\rho_{f,\hspace{0.7pt}\fx}\subseteq(-\alpha-2,-\alpha)$, then, $\forall k \in \bZ$, 
\[V^ki_{f,+}\scM(f^{-\alpha})_{(\fx,0)} = V^{k+1}_{\text{\emph{ind}}}i_{f,+}\scM(f^{-\alpha})_{(\fx,0)} + \beta_{f,-\alpha,\hspace{0.7pt}\fx}(\partial_tt-k+\alpha)\cdot V^k_{\text{\emph{ind}}}i_{f,+}\scM(f^{-\alpha})_{(\fx,0)},\]
where
\[\beta_{f,-\alpha,\hspace{0.7pt}\fx}(s)=\prod_{\lambda \in \rho_{f,\hspace{0.7pt}\fx}\cap (-\alpha-1,-\alpha)}(s+\lambda+1)^{l_{\lambda}},\]
where $l_{\lambda}$ is the multiplicity of $\lambda$ as a root of $b_{f,\hspace{0.7pt}\fx}(s)$.

\label{propVfil}
    
\end{prop}

\begin{proof}

Firstly, see that $\beta_{f,-\alpha,\hspace{0.7pt}\fx}(s)|b_{f,\hspace{0.7pt}\fx}(-s-1)$, so we may write 
\[b_{f,\hspace{0.7pt}\fx}(-s-1)=\beta_{f,-\alpha,\hspace{0.7pt}\fx}(s) \widetilde{\beta}_{f,-\alpha,\hspace{0.7pt}\fx}(s),\]
where the roots of $\widetilde{\beta}_{f,-\alpha,\hspace{0.7pt}\fx}(s)$ lie in the interval $[\alpha,\alpha+1)$. Now define 
\[\overline{V}^ki_{f,+}\scM(f^{-\alpha})_{(\fx,0)}:= V_{\text{ind}}^{k+1}i_{f,+}\scM(f^{-\alpha})_{(\fx,0)}+\beta_{f,-\alpha,\hspace{0.7pt}\fx}(\partial_tt-k+\alpha)V_{\text{ind}}^ki_{f,+}\scM(f^{-\alpha})_{(\fx,0)}\]
for each $k \in \mathbb{Z}$. Write also
\[\overline{\beta}_{f,-\alpha,\hspace{0.7pt}\fx}(s) := \beta_{f,-\alpha,\hspace{0.7pt}\fx}(s-1)\widetilde{\beta}_{f,-\alpha,\hspace{0.7pt}\fx}(s).\]
See that $\overline{V}^{\bullet}$ is a good $V$-filtration. We show that $\overline{\beta}_{f,-\alpha,\hspace{0.7pt}\fx}(s+\alpha)$ is a B-function for the good $V$-filtration $\overline{V}^{\bullet}$, via the following calculations.\vspace{-10pt}

\begin{align*} 
\overline{\beta}_{f,-\alpha,\hspace{0.7pt}\fx}(\partial_tt-k+\alpha)\cdot V^{k+1}_{\text{ind}} &= \widetilde{\beta}_{f,-\alpha,\hspace{0.7pt}\fx}(\partial_tt-k+\alpha)\beta_{f,-\alpha,\hspace{0.7pt}\fx}(\partial_tt-(k+1)+\alpha)\cdot V^{k+1}_{\text{ind}} \\
&\subseteq \widetilde{\beta}_{f,-\alpha,\hspace{0.7pt}\fx}(\partial_tt-k)\cdot \overline{V}^{k+1} \;\;\; \text{by defn of } \overline{V}^{k+1}\\ 
&\subseteq \overline{V}^{k+1}\\
\end{align*}

\vspace{-20pt}\noindent and 
\begin{align*} 
\overline{\beta}_{f,-\alpha,\hspace{0.7pt}\fx}(\partial_tt-k+\alpha)\cdot [ \beta_{f,-\alpha,\hspace{0.7pt}\fx}(\partial_tt&-k+\alpha)\cdot V^k_{\text{ind}}] \\&= \beta_{f,-\alpha,\hspace{0.7pt}\fx}(\partial_tt-k+\alpha-1)\cdot \left[b_{f,\hspace{0.7pt}\fx}(-(\partial_tt-k+\alpha)-1)\cdot V^k_{\text{ind}}\right] \\ 
&\subseteq \beta_{f,-\alpha,\hspace{0.7pt}\fx}(\partial_tt-(k+1)+\alpha)\cdot V^{k+1}_{\text{ind}} \;\;\;\; (*) \\ 
&\subseteq \overline{V}^{k+1} \;\;\;\; \text{as above,}
\end{align*}

\noindent where ($*$) follows since $b_{f,\hspace{0.7pt}\fx}(-s-1-\alpha)$ is the Bernstein-Sato polynomial for the good $V$-filtration $V_{\text{ind}}^{\bullet}$, by Lemma \ref{lemBfcn}. 

These two inclusions imply that 
\[\overline{\beta}_{f,-\alpha,\hspace{0.7pt}\fx}(\partial_tt-k+\alpha)\cdot \overline{V}^k \subseteq \overline{V}^{k+1}\] 
for every $k \in \mathbb{Z}$, i.e., that $\overline{\beta}$ is a B-function for the good $V$-filtration $\overline{V}^{\bullet}$. But, the roots of $\overline{\beta}$ are contained within the half-open interval $[\alpha, \alpha+1)$. This implies further that the Bernstein-Sato polynomial of $\overline{V}^{\bullet}$ has its roots in the interval $[0,1)$. Therefore, by the uniqueness of such a good $V$-filtration (see \cite{MM04}, Proposition 4.2-6), $\overline{V}^k = V^k$ for every $k \in \mathbb{Z}$, as required.
    
\end{proof}

\begin{cor}

If $\rho_{f,\hspace{0.7pt}\fx}\subseteq(-\alpha-2,-\alpha)$, then, $\forall k \in \bZ_{\geq 0}$, 
\[V^{k+\alpha}i_{f,+}\scO_X(*f)_{(\fx,0)} = V^{k+1}_{\text{\emph{ind}}}i_{f,+}\scO_X(*f)_{(\fx,0)} + \beta_{f,-\alpha,\hspace{0.7pt}\fx}(\partial_tt-k)V^k_{\text{\emph{ind}}}i_{f,+}\scO_X(*f)_{(\fx,0)},\]
where $V_{\text{\emph{ind}}}^ki_{f,+}\scO_X(*f)_{(\fx,0)}:=V^k\scD_{X\times\bC,(\fx,0)}\cdot f^{-1}_{\fx}$ (which is non-exhaustive if $-2\in\rho_{f,\hspace{0.7pt}\fx}$).

\label{corVfil2}
    
\end{cor}

\begin{proof}

This follows immediately by Proposition \ref{propiso}, since the isomorphism $\Phi$ defined in this proposition is $s$-linear and $t$-linear, and since 
\[\Phi(V^ki_{f,+}\scM(f^{-\alpha})) = V^{k+\alpha}i_{f,+}\scO_X(*f).\]
    
\end{proof}

As in the case $\alpha=0$, the $V$-filtration gives us an expression for the Hodge filtration on $\scM(f^{-\alpha})$. First we define some linear maps we will use in the statement of this theorem.

\begin{deflem}

Consider the following maps.
\begin{align*}
\varphi_1:i_{f,+}\scM(f^{-\alpha})\to \scM(f^{-\alpha})[s]f^s\,\, &; \,\, uf^{-\alpha}\partial_t^i \mapsto uf^{-i-\alpha}Q_i(-s)f^s,\\
\varphi_2:i_{f,+}\scO_X(*f)\to \scO_X(*f)[s]f^s\,\, &; \,\, u\partial_t^i \mapsto uf^{-i}Q_i(-s)f^s,\\
\Phi':\scM(f^{-\alpha})[s]f^s\to\scO_X(*f)[s]f^s\,\, &; \,\, g(s)f^{-\alpha}f^s\mapsto g(s+\alpha)^if^s,\\
\phi_{-\alpha}:\scO_X(*f)[s]f^s\to \scM(f^{-\alpha})\,\, &; \,\, g(s)f^s \mapsto g(-\alpha)f^{-\alpha},\\
\phi_0:\scM(f^{-\alpha})[s]f^s\to \scM(f^{-\alpha})\,\, &; \,\, g(s)f^{s-\alpha} \mapsto g(0)f^{-\alpha},\\
\pi_{f,-\alpha}:\scD_X[s]\to i_{f,+}\scM(f^{-\alpha})\,\, &; \,\, P(s)\mapsto P(s)\cdot f^{-1-\alpha}.\\
\end{align*}

\vspace{-10pt}\noindent Then $\varphi_1$, $\varphi_2$ and $\Phi'$ are all isomorphisms of $\scD_X\langle t,t^{-1},s\rangle$-modules\footnote{Where $s$ acts as $s+\alpha$ on the right hand side of the map $\Phi'$, as with $\Phi$.}, and $\Phi = \varphi_2^{-1}\circ\Phi'\circ\varphi_1$. Moreover, $\pi_{f,-\alpha}$ has kernel $\text{\emph{ann}}_{\scD_X[s]}f^{s-1-\alpha}$ and when $\rho_{f,\hspace{0.7pt}\fx}\cap(\bZ_{<-1}-\alpha)=\emptyset$, $\pi_{f,-\alpha}$ has image $V^0_{\text{\emph{ind}}}i_{f,+}\scM(f^{-\alpha})_{(\fx,0)}$ at $\fx$. Note that we will use the same notation for the induced maps between localisations.

\label{deflem}

\end{deflem}

\begin{proof}

See \cite{MP20}. For the final statement, see that $\pi_{f,-\alpha}$ is the composition of the map $\scD_X[s]\to\scD_X[s]\cdot f^{s-1-\alpha}\subseteq \scM(f^{-\alpha})[s]f^s$ with the isomorphism $\varphi_1^{-1}$.
    
\end{proof}

\begin{note}

We are abusing notation here, in that we already defined $\phi_{-\alpha}$ in the introduction as the specialisation map $\scD_X[s]\to\scD_X$. We will continue this abuse of notation through the remainder of the article and hope no confusion is caused.
    
\end{note}

The following is simply a rewording of a result in \cite{MP20}. 

\begin{thm}

For each non-negative integer $k$ and $\alpha \in \bQ$,
\begin{center} \vspace{-0.4cm} \begin{align*}F_k^H\scM(f^{-\alpha}) & = \left(\partial_tF_{k-1}^{t-\text{\emph{ord}}}i_{f,+}\scM(f^{-\alpha})+V^0i_{f,+}\scM(f^{-\alpha})\right)\cap\scM(f^{-\alpha}) \\ &= \psi_0\left(F_k^{t-\text{\emph{ord}}}V^0i_{f,+}\scM(f^{-\alpha})\right)\\
&=\psi_{-\alpha}\left(F_k^{t-\text{\emph{ord}}}V^{\alpha}i_{f,+}\scO_X(*f)\right),\end{align*}\end{center}
where $\psi_0 := \phi_0\circ \varphi_1$ and $\psi_{-\alpha}:=\phi_{-\alpha}\circ \varphi_2$. Here, \vspace{-5pt}
\[F_k^{t-\text{\emph{ord}}}i_{f,+}\scM(f^{\beta}):=\sum_{i=0}^k\scM(f^{\beta})\partial_t^i\vspace{-5pt}\]
for any $\beta \in \bQ$.

\label{thmformulaQ}

\end{thm}

\begin{proof}

The final two equalities are immediate, since $\psi_0=\psi_{-\alpha}\circ \Phi$ and since $\Phi$ preserves the $t$-order filtration. Thus it suffices to prove that 
\[F_k^H\scM(f^{-\alpha})=\psi_{-\alpha}\left(F_k^{t-\text{ord}}V^{\alpha}i_{f,+}\scO_X(*f)\right).\]
For $\alpha=0$ this was proven in \cite{BD24} (although it also follows from the $\alpha=1$ case, for instance). For $\alpha> 0$, this follows immediately by \cite{MP20}, Theorem A', since, if $v=\sum_{i=0}^kv_j\partial_t^j \in i_{f,+}\scO_X(*f)$, then $\psi_{-\alpha}(v)=\sum_{i=0}^kQ_j(\alpha)f^{-j-\alpha}v_j$.
    
\end{proof}

\section{The main theorem} \label{sectionmainthm}

In this section we combine the formulae for the Hodge filtration and for the $V$-filtration, under our hypotheses, in order to obtain the expression given in the main theorem in the introduction. We first assume that $\rho_{f,\hspace{0.7pt}\fx}\subseteq(-2-\alpha,-\alpha)$ and then later assume in addition that $f$ is parametrically prime and Euler homogeneous at $\fx\in X$. The proofs are completely analogous to those used in \cite{BD24}, which are in turn mostly analogous to those of \cite{CNS22}.

\begin{proof}[Proof of Theorem \ref{thmmain2}]

By Proposition \ref{propVfil}, it is easy to see that
\[\pi_{f,-\alpha}(\widetilde{\Gamma}_{f,-\alpha,\hspace{0.7pt}\fx})=V^0i_{f,+}\scM(f^{-\alpha})_{(\fx,0)},\]
where
\[\widetilde{\Gamma}_{f,-\alpha,\hspace{0.7pt}\fx}:=\scD_{X,\hspace{0.7pt}\fx}[s]f_{\fx}+\scD_{X,\hspace{0.7pt}\fx}[s]\beta_{f,-\alpha,\hspace{0.7pt}\fx}(-s+\alpha)+\text{ann}_{\scD_{X,\hspace{0.7pt}\fx}[s]}f_{\fx}^{s-1-\alpha}\subseteq\scD_{X,\hspace{0.7pt}\fx}[s].\]
So if $uf_{\fx}^{-\alpha}\in F_0^H\scM(f^{-\alpha})_{\fx}$, $u\in\scO_X(*f)_{\fx}$, we may by Theorem \ref{thmformulaQ} write $uf_{\fx}^{-\alpha}=\pi_{f,-\alpha}(\gamma)$ with $\gamma\in \widetilde{\Gamma}_{f,-\alpha,\hspace{0.7pt}\fx}$. Since $F_0^H \scM(f^{-\alpha})_{\fx} \subseteq \scO_{X,\hspace{0.7pt}\fx} \cdot f_{\fx}^{-1-\alpha}$ (by Theorem \ref{thmformulaQ} for instance), we know that $uf_{\fx} \in \scO_{X,\hspace{1pt}\fx}$ and that $\gamma - uf_{\fx} \in \scD_{X,\hspace{1pt}\fx}[s]$. Then $\pi_{f,-\alpha}(\gamma-uf_{\fx})=0$, implying that $uf_{\fx} \in \gamma+\ker\pi_{f,-\alpha} \subseteq \widetilde{\Gamma}_{f,-\alpha,\hspace{0.7pt}\fx}$, since $\ker\pi_{f,-\alpha}=\text{ann}_{\scD_{X,\hspace{1pt}\fx}[s]}f_{\fx}^{s-1-\alpha}\subseteq\widetilde{\Gamma}_{f,-\alpha,\hspace{0.7pt}\fx}$ as seen in the statement of Definition/Lemma \ref{deflem}. Thus $uf_{\fx}^{-\alpha} = (uf_{\fx})f_{\fx}^{-1-\alpha} \in (\widetilde{\Gamma}_{f,-\alpha,\hspace{0.7pt}\fx} \cap \scO_{X,\hspace{1pt}\fx}) f_{\fx}^{-1-\alpha} = (\Gamma_{f,-\alpha,\hspace{0.7pt}\fx} \cap \scO_{X,\hspace{1pt}\fx}) f_{\fx}^{-1-\alpha}$, proving that $F_0^H \scM(f^{-\alpha})_{\fx} \subseteq (\Gamma_{f,-\alpha,\hspace{0.7pt}\fx} \cap \scO_{X,\hspace{1pt}\fx}) f^{-1-\alpha}_{\fx}$ as required.

Conversely, take $u \in \Gamma_{f,-\alpha,\hspace{0.7pt}\fx} \cap \scO_{X,\hspace{1pt}\fx}=\widetilde{\Gamma}_{f,-\alpha,\hspace{0.7pt}\fx} \cap \scO_{X,\hspace{1pt}\fx}$. Then $u f^{-1-\alpha}_{\fx} = \pi_{f,-\alpha}(u) \in V^0i_{f,+}\scM(f^{-\alpha})_{(\fx,0)}$, implying that $u f^{-1-\alpha}_{\fx} \in F_0^H \scM(f^{-\alpha})_{\fx}$ by Theorem \ref{thmformulaQ}.

\end{proof}

\begin{prop}

Assume that $f$ is parametrically prime and Euler homogeneous at $\fx$ and that $\rho_{f,\hspace{0.7pt}\fx}\cap(\bZ_{<-1}-\alpha)=\emptyset$. Let $k,l\in\bZ_{\geq 0}$. Then
\[F_l^{\text{\emph{ord}}}i_{f,+}\scM(f^{-\alpha})_{(\fx,0)}\cap V^k_{\text{\emph{ind}}}i_{f,+}\scM(f^{-\alpha})_{(\fx,0)} = \left(F_l\scD_{X\times\bC,(\fx,0)}\cap V^k\scD_{X\times \bC,(\fx,0)}\right)\cdot f_{\fx}^{-1-\alpha}.\]

\label{propcomp}
    
\end{prop}

\begin{proof}

Exactly as in \cite{CNS22} and \cite{BD24}, using a lemma of Saito (Lemma 1.2.14 of \cite{MSai88}), and Lemma \ref{lemgen}, we reduce to showing that the natural map
\[F_l\scD_{X\times\bC,(\fx,0)}\cap V^k\scD_{X\times\bC,(\fx,0)}\cap I_{\alpha}(f_{\fx})\to \sigma_l(V^k\scD_{X\times\bC,(\fx,0)})\cap\sigma_l(I_{\alpha}(f_{\fx}))\]
is surjective, where $\sigma_l$ is the principal symbol map of order $l$ and where 
\[I_{\alpha}(f_{\fx}):=\scD_{X\times\bC,(\fx,0)}(t-f_{\fx})+\scD_{X\times\bC,(\fx,0)}(E+\partial_tt+1+\alpha)+\text{ann}_{\scD_{X,\hspace{0.7pt}\fx}}f_{\fx}^s\subseteq\scD_{X\times\bC,(\fx,0)},\]
coinciding by Lemma \ref{lemann} with $\text{ann}_{\scD_{X\times\bC,(\fx,0)}}f_{\fx}^{-1-\alpha}$, so that
\[i_{f,+}\scM(f^{-\alpha})_{(\fx,0)}\simeq \scD_{X\times\bC,(\fx,0)}/I_{\alpha}(f_{\fx}).\]

Choose firstly $\zeta_1,\ldots,\zeta_m$ to be a Gröbner basis of $\text{ann}_{\scD_{X,\hspace{0.7pt}\fx}}f_{\fx}^s$.

Now, given $P\in I_{\alpha}(f_{\fx})$ of order $l$ such that $\sigma(P)\in\sigma(V^k\scD_{X\times\bC,(\fx,0)})$, it is shown (under the assumption that $f$ is parametrically prime and Euler homogeneous at $\fx$) in the proof of \cite{BD24}, Proposition 4.21, that we may write
\[\sigma(P)=\sum_{i=0}^{m+1}k_i\sigma(G_i),\]
where $G_0 = t-f_{\fx}$, $G_i = \zeta_i$ for $1\leq i \leq m$ and $G_{m+1}=E+\partial_tt+1+\alpha$, and where 
\[k_i \in V^k\text{gr}^F_{l-\text{ord}(G_i)}\scD_{X\times\bC,(\fx,0)}\]
for all $i$.

Then choose lifts $K_i\in F_{l-\text{ord}(G_i)}\scD_{X\times\bC,(\fx,0)}\cap V^k\scD_{X\times\bC,(\fx,0)}$ of $k_i$, and write
\[P' := \sum_{i=0}^{m+1}K_iG_i \in I_{\alpha}(f_{\fx}).\]
See that $P' \in F_l\scD_{X\times \bC,(\fx,0)}\cap V^k\scD_{X\times \bC,(\fx,0)} \cap I_{\alpha}(f_{\fx})$ and that $\sigma(P)=\sigma(P')$, thus showing the required surjectivity.

\end{proof}

\begin{lem}

Assume that $f$ is parametrically prime and Euler homogeneous at $\fx\in X$ and that $\rho_{f,\hspace{0.7pt}\fx}\subseteq (-2-\alpha,-\alpha)$. Then
\[F_{k+1}^HV^0i_{f,+}\scM(f^{-\alpha})_{(\fx,0)}=F_k^{t-\text{\emph{ord}}}V^0i_{f,+}\scM(f^{-\alpha})_{(\fx,0)}= F_k^{\text{\emph{ord}}}V^0i_{f,+}\scM(f^{-\alpha})_{(\fx,0)}.\]

\label{lemincl2}
    
\end{lem}

\begin{proof}

As in the case $\alpha=0$, the first equality actually holds in full generality. We prove this via induction on $k$. The inclusion $\subseteq$ is trivial. Given 
\[u=\sum_{j=0}^ku_j\partial_t^j \in V^0i_{f,+}\scM(f^{-\alpha})_{(\fx,0)}\]
(with $u_j\in\scM(f^{-\alpha})_{\fx}$ for all $j$), $u_0 = \psi_0(u)\in F_k^H\scM(f^{-\alpha})_{\fx}$ by Theorem \ref{thmformulaQ}. If \underline{$k=0$}, then $u=u_0$, so 
\[u \in F_0^H\scM(f^{-\alpha})_{\fx} = F_1^Hi_{f,+}\scM(f^{-\alpha})_{(\fx,0)},\]
showing the desired result. If \underline{$k>0$}, then 
\[\sum_{j=1}^kju_j\partial_t^{j-1}=(f_{\fx}-t)\cdot u \in F_{k-1}^{t-\text{ord}}V^0i_{f,+}\scM(f^{-\alpha})_{(\fx,0)},\]
implying by the inductive hypothesis that
\[\sum_{j=1}^kju_j\partial_t^{j-1} \in F_k^Hi_{f,+}\scM(f^{-\alpha})_{(\fx,0)},\]
which in turn implies that $u_j \in F_{k-j}^H\scM(f^{-\alpha})_{\fx}$ for all $j \geq 1$. Since we already know that $u_0 \in F_k^H\scM(f^{-\alpha})_{\fx}$, 
\[u=\sum_{j=0}^ku_j\partial_t^j \in \sum_{j=0}^kF_{k-j}^H\scM(f^{-\alpha})_{\fx}\partial_t^j = F_{k+1}^Hi_{f,+}\scM(f^{-\alpha})_{(\fx,0)},\]
proving the required inclusion.

Now we prove the second equality. The inclusion $\supseteq$ is clear. If $u \in F_k^{t-\text{ord}}V^0i_{f,+}\scM(f^{-\alpha})_{(\fx,0)}$, then, as in the proof of \cite{BD24}, Proposition 4.22, there exists some $p\geq 0$ such that $t^p \cdot u \in F_k^{\text{ord}}V^pi_{f,+}\scM(f^{-\alpha})_{(\fx,0)}$.

Now Proposition \ref{propVfil} implies that $V^pi_{f,+}\scM(f^{-\alpha})_{(\fx,0)} \subseteq V^p_{\text{ind}}i_{f,+}\scM(f^{-\alpha})_{(\fx,0)}$, so by Proposition \ref{propcomp} there exists some $P\in F_k\scD_{X\times\bC,(\fx,0)}\cap V^p\scD_{X\times \bC,(\fx,0)}$ such that $P\cdot f_{\fx}^{-1-\alpha} = t^p\cdot u$. 

Finally, we may by the definition of $V^{\bullet}\scD_{X\times \bC,(\fx,0)}$ write $P=t^p\cdot P'$, where $P'\in F_k\scD_{X\times\bC,(\fx,0)}\cap V^0\scD_{X\times \bC,(\fx,0)}$. Then, by the injectivity of $t$, 
\[u= P'\cdot f_{\fx}^{-1-\alpha} \in F_k^{\text{ord}}i_{f,+}\scM(f^{-\alpha})_{(\fx,0)}.\]

\end{proof}

\begin{thm}

Assume that $f$ is parametrically prime and Euler homogeneous at $\fx\in X$ and that $\rho_{f,\hspace{0.7pt}\fx}\subseteq (-2-\alpha,-\alpha)$. Then
\[F_k^H\scM(f^{-\alpha})_{\fx}\subseteq F_k^{\text{\emph{ord}}}\scM(f^{-\alpha})_{\fx}.\]

\label{thmincl}
    
\end{thm}

\begin{proof}

Let $u_0 \in F_k^H \scM(f^{-\alpha})_{\fx}$. By Lemma \ref{lemincl2} and Theorem \ref{thmformulaQ}, we find $u = u_0  + v \in F_{k+1}^{\text{ord}} V^0 i_{f,+} \scM(f^{-\alpha})_{(\fx,0)}$, where $v = \sum_{j \geq 1} v_j \partial_t^j$ for $v_j \in \scM(f^{-\alpha})_{\fx}$. Write $u = P \cdot f_{\fx}^{-1-\alpha}$ where $P = \sum_{j\geq 0} P_{k-j} \partial_t^j$ and $P_{k-j} \in F_{k-j} \scD_{X,\hspace{0.7pt}\fx}$. Then
\[u_0 = P_k\cdot f_{\fx}^{-1-\alpha} \in F_k^{\text{ord}}\scM(f^{-\alpha})_{\fx}.\]
    
\end{proof}

\begin{lem}

Assume that $f$ is parametrically prime and Euler homogeneous at $\fx\in X$ and that $\rho_{f,\hspace{0.7pt}\fx}\subseteq (-2-\alpha,-\alpha)$. Then
\[\pi_{f,-\alpha}\left(\widetilde{\Gamma}_{f,-\alpha,\hspace{0.7pt}\fx}\cap F_k^{\sharp}\scD_{X,\hspace{0.7pt}\fx}[s]\right)=F_k^{\text{\emph{ord}}}V^0i_{f,+}\scM(f^{-\alpha})_{(\fx,0)}\]
for all $k \geq 0$, where
\[\widetilde{\Gamma}_{f,-\alpha,\hspace{0.7pt}\fx}:=\scD_{X,\hspace{0.7pt}\fx}[s]f_{\fx}+\scD_{X,\hspace{0.7pt}\fx}[s]\beta_{f,-\alpha,\hspace{0.7pt}\fx}(-s+\alpha)+\text{\emph{ann}}_{\scD_{X,\hspace{0.7pt}\fx}[s]}f_{\fx}^{s-1-\alpha}\subseteq\scD_{X,\hspace{0.7pt}\fx}[s].\]

\end{lem}

\begin{proof}

Firstly, as seen above, it is clear that 
\[\pi_{f,-\alpha}\left(\widetilde{\Gamma}_{f,-\alpha,\hspace{0.7pt}\fx}\right) = V^0i_{f,+}\scM(f^{-\alpha}),\]
by Proposition \ref{propVfil}. It is also clear by the linearity of $\pi_{f,-\alpha}$ that $\pi_{f,-\alpha}(F_k^{\sharp}\scD_{X,\hspace{0.7pt}\fx}[s])\subseteq F_k^{\text{ord}}i_{f,+}\scM(f^{-\alpha})_{(\fx,0)}$. Therefore it suffices to prove that
\[F_k^{\text{ord}}V^0i_{f,+}\scM(f^{-\alpha})_{(\fx,0)}\subseteq \pi_{f,-\alpha}\left(\widetilde{\Gamma}_{f,-\alpha,\hspace{0.7pt}\fx}\cap F_k^{\sharp}\scD_{X,\hspace{0.7pt}\fx}[s]\right).\]

Assume that $u\in F_k^{\text{ord}}V^0i_{f,+}\scM(f^{-\alpha})_{(\fx,0)}$. Then we may write $u=\pi_{f,-\alpha}(Q)$ for some $Q\in\widetilde{\Gamma}_{f,-\alpha,\hspace{0.7pt}\fx}$. Also, since $V^0i_{f,+}\scM(f^{-\alpha})_{(\fx,0)}\subseteq V^0_{\text{ind}}i_{f,+}\scM(f^{-\alpha})_{(\fx,0)}$ by Proposition \ref{propVfil}, and since the order filtration and induced $V$-filtration on $i_{f,+}\scM(f^{-\alpha})_{(\fx,0)}$ are compatible in the sense of Proposition \ref{propcomp}, we may also write $u=\pi_{f,-\alpha}(\widetilde{P})$ for some $\widetilde{P}\in F_k^{\sharp}\scD_{X,\hspace{0.7pt}\fx}[s]$.

Then see that 
\[\widetilde{P}-Q\in\ker\pi_{f,-\alpha}=\text{ann}_{\scD_{X,\hspace{0.7pt}\fx}[s]}f_{\fx}^{s-1-\alpha}\subseteq \widetilde{\Gamma}_{f,-\alpha,\hspace{0.7pt}\fx},\]
so that in fact $\widetilde{P}\in \widetilde{\Gamma}_{f,-\alpha,\hspace{0.7pt}\fx}$ also. This shows the required inclusion.

\end{proof}

\begin{proof}[Proof of Theorem \ref{thmmain}]

The main theorem now follows by combining what we know about the filtrations on $V^0i_{f,+}\scM(f^{-\alpha})_{(\fx,0)}$.

\begin{align*}
F_k^H\scM(f^{-\alpha})_{\fx} &= \psi_0\left(F_k^{t-\text{ord}}V^0i_{f,+}\scM(f^{-\alpha})_{(\fx,0)}\right) \\
&= \psi_0\left(F_k^{\text{ord}}V^0i_{f,+}\scM(f^{-\alpha})_{(\fx,0)}\right)\\
&= \psi_0\left(\pi_{f,-\alpha}\left(\widetilde{\Gamma}_{f,-\alpha,\hspace{0.7pt}\fx}\cap F_k^{\sharp}\scD_{X,\hspace{0.7pt}\fx}[s]\right)\right)\\
&= \psi_0\left(\varphi_1^{-1}\left((\widetilde{\Gamma}_{f,-\alpha,\hspace{0.7pt}\fx}\cap F_k^{\sharp}\scD_{X,\hspace{0.7pt}\fx}[s])\cdot f_{\fx}^{s-1-\alpha}\right)\right)\\
&= \phi_0\left((\widetilde{\Gamma}_{f,-\alpha,\hspace{0.7pt}\fx}\cap F_k^{\sharp}\scD_{X,\hspace{0.7pt}\fx}[s])\cdot f_{\fx}^{s-1-\alpha}\right)\\
&= \phi_0\,(\widetilde{\Gamma}_{f,-\alpha,\hspace{0.7pt}\fx}\cap F_k^{\sharp}\scD_{X,\hspace{0.7pt}\fx}[s])\cdot f_{\fx}^{-1-\alpha}\\
&= \phi_{-\alpha}\,(\Gamma_{f,-\alpha,\hspace{0.7pt}\fx}\cap F_k^{\sharp}\scD_{X,\hspace{0.7pt}\fx}[s])\cdot f_{\fx}^{-1-\alpha}.\\
\end{align*}

\end{proof}

\bibliographystyle{siam}
\bibliography{bibliography.bib}

\begin{thebibliography}{1}

\bibitem{BD24}
{\sc D.~Bath and H.~Dakin}, {\em The {H}odge filtration and parametrically prime divisors}, 2024.
\newblock Preprint, arXiv:2408.02601.

\bibitem{CNS22}
{\sc A.~Castaño~Domínguez, L.~Narváez~Macarro, and C.~Sevenheck}, {\em Hodge ideals of free divisors}, Selecta Mathematica, 28 (2022).

\bibitem{Kash76}
{\sc M.~Kashiwara}, {\em B-functions and holonomic systems}, Invent. Math., 38 (1976), pp.~33--53.

\bibitem{MM04}
{\sc P.~Maisonobe and Z.~Mebkhout}, {\em Le théorème de comparaison pour les cycles évanescents}, Eléments de la théorie des systèmes différentiels géométriques, Sémin. Congr., 8 (2004), pp.~311--389.

\bibitem{MP20}
{\sc M.~Musta{\c t}{\u a} and M.~Popa}, {\em Hodge ideals for {Q}-divisors, {V}-filtration, and minimal exponent}, For. Math. Sigma, 8 (2020).

\bibitem{Nar15}
{\sc L.~Narváez~Macarro}, {\em A duality approach to the symmetry of {B}ernstein–{S}ato polynomials of free divisors}, Advances in Mathematics, 281 (2015), p.~1242–1273.

\bibitem{MSai88}
{\sc M.~Saito}, {\em Modules de {H}odge polarisables}, Publ. Res. Inst. Math. Sci., 24 (1988), pp.~849--995.

\bibitem{Wal17}
{\sc U.~Walther}, {\em The {J}acobian module, the {M}ilnor fiber, and the {D}-module generated by $f^s$}, Invent. Math., 207 (2017), pp.~1239--1287.

\end{thebibliography}

\end{document}